\documentclass[12pt]{article}%
\usepackage{amsmath,amssymb,amsfonts,amscd,amsthm}
\usepackage{amsfonts}
\usepackage{amsmath}
\usepackage{graphicx}%
\usepackage{float}
\usepackage[margin=1.5cm]{caption}
\usepackage{subcaption}
\usepackage{epstopdf}
\usepackage{xcolor}

\usepackage{tikz}
\usetikzlibrary{arrows,shapes,positioning}
\usetikzlibrary{decorations.markings}
\tikzstyle arrowstyle=[scale=1]
\tikzstyle directed=[postaction={decorate,decoration={markings,
    mark=at position .65 with {\arrow[arrowstyle]{stealth}}}}]
\tikzstyle reverse directed=[postaction={decorate,decoration={markings,
    mark=at position .65 with {\arrowreversed[arrowstyle]{stealth};}}}]
    
\providecommand{\U}[1]{\protect\rule{.1in}{.1in}}
\newtheorem{theorem}{Theorem}

\newtheorem{lemma}[theorem]{Lemma}

\newtheorem{prop}[theorem]{Proposition}

\begin{document}
  \begin{center}
        {\fontsize{18}{22}\selectfont
       \bf Choreography solutions of the $n$-body problem on $S^2$}
       \end{center}
       
       \vspace{4mm}

        \begin{center}
        {\bf Juan Manuel S\'anchez-Cerritos$^1$ and Shiqing Zhang$^1$}\\
\bigskip
$^1$Department of Mathematics\\
Sichuan University, Chengdu, People's Republic of China\\
\bigskip
sanchezj01@gmail.com, zhangshiqing@msn.com
       \end{center}

\abstract{We try to prove the existence of choreography solutions for the $n-$body problem on $S^2$}. For the three-body problem, we show the existence of the 8-shape orbit on $S^2$.\\



\section{Introduction}

The curved $n-$body problem is a generalization of the Newtonian gravitational problem. It has been studied for several authors, for example in \cite{Diacu4,Diacu5,Diacu6,Diacu3,Diacu,Diacu2,paper1,naranjo,paper,Zhu}. Particularly the interesting history of this problem can be found on \cite{Diacu4, Diacu5}. Here we consider the positive curvature case, i.e. particles moving on the unit sphere, $S^2$.

The motion of the $n$ particles with masses $m_i>0$ and positions $q_i \in S^2$, $i=1\dots,n$,  is described by the following system

\begin{equation}
 m_i\ddot{q}_i=\frac{\partial U}{\partial q_i}-m_i(\dot{q}_i\cdot \dot{q_i})q_i, \ \ \ i=1\dots,n, \label{eq}
\end{equation}
where $U$ is the force function which generalizes the Newtonian one, and it is given by

\begin{equation}
U=\sum_{i<j}m_i m_j\cot(d(q_i,q_j)). 
\end{equation}

On classical n-body problems, Chenciner and Montgomery proved the existence of the eight-shape choreography for the three body problem in 2000 \cite{chenciner}, which was described numerically by Moore in 1993 \cite{moore} and Sim\'o in 2000 \cite{Simo}. In the recent years Montenelli and Gushterov computed numerically the analogue solutions in the positive curved space \cite{montanelli}.

The goal of this work is, based on the work of Zhang and co-authors \cite{Zhang1, Zhang2, Zhang3}, to prove the existence of periodic solutions for the $n-$body problem on $S^2$.



Motivated by Sim\'o \cite{Simo} for planar $N-$body problems, in this paper we seek for periodic solutions of (\ref{eq}) moving on the same orbit, i.e., setting the period as $T=1$, we look for solutions such that 

\begin{equation}
q_i(t)=Q\left(t+k_i \right), \ \ \ i=1,\dots,n,
\end{equation}
with $\ 0=k_1<\dots<k_n<1$, and for some periodic function $Q:[0,1]\rightarrow S^2$.

Since our problem is on the curved space, the solution is much more complicated than Euclidean space.

We define the following set

\begin{equation}
\begin{split}
D=&\left\lbrace q=(q_1,\dots,q_n) \in (S^2)^n \ | \ q_i  \ \text{is absolutely continuous and} \ q_i(t)\neq q_j(t), \right. \\
 &\left. \text{for} \ \ 1 \leq i\neq j \leq n  \right\rbrace.
\end{split}
\end{equation}


The Lagrangian action associated to system (\ref{eq}) on $D$ is

\begin{equation}
f(q)=\int_0^1\left( \dfrac{1}{2}\sum_1^nm_i|\dot{q}_i(t)|^2+U(q(t)) \right) dt. \label{functional}
\end{equation}

We are interested in showing the existence of new choreography solutions of (\ref{eq}). In other words, we will not only show that the Lagrangian action functional reaches its minimum in $D$, but in a subset where the $n$ particles follow the same orbit.

There are some works where circular choreography solutions have been found, see for instance \cite{Diacu,paper1}. In order to find new families of choreographies we will introduce the following sets

\begin{equation}
\begin{split}
E_1=&\left\lbrace q=(q_1,\dots,q_n) \in D \ | \ q_1(t)=q_n(t+1/n), \  \ q_i(t)=q_{i-1}(t+1/n),\right.\\
&\left. i=2,\dots, n \ \right\rbrace, \nonumber\\
E_2=&\left\lbrace q=(q_1,\dots,q_n) \in D \ | \ q_1(t+1/2)=diag\{1,-1,1 \} q_1(t)\  \right\rbrace, \nonumber\\
E_3=&\left\lbrace q=(q_1,\dots,q_n) \in D \ | \ q_1(-t)=diag\{ -1,-1,1 \} q_1(t)\ \right\rbrace.\nonumber
\end{split}
\end{equation}

%

It is not difficult to see that $q_1(0)=(0,0,1)=q_1(1/2)$ for $q \in E_2 \cap E_3$ . Hence circular orbits mentioned above do not belong to $E_2 \cap E_3$. The set of choreographies are orbits on

\[  H=\left\lbrace q=(q_1,\dots,q_n) \in D \ | q_1 \in E_1 \cap E_2 \cap E_3\ \right\rbrace.  \]

Let $B=diag\{1,-1,1 \}$ and $C=diag\{-1,-1,1 \}$. We now define the following actions $\Phi_1$, $\Phi_2$ and $\Phi_3$ on $D$

 \[  \Phi_1(q(t))=(q_n(t+1/n),q_1(t+1/n),\cdots, q_{n-1}(t+1/n)), \]
 \[  \Phi_2(q(t))=(Bq_1(t+1/2),Bq_2(t+1/2),Bq_3(t+1/2)), \]
 \[ \Phi_3(q(t))=(Cq_1(-t),Cq_2(-t),Cq_3(-t)) . \]

Then the fixed point of $\Phi_i$ is $E_i$, $i=1,2,3$. 
 
 We refer to Palais' principle of symmetric criticality \cite{palais}, in order to conclude that the critical points of $f$ restricted to $H$ are critical points of $f$ on $D$ as well. 
 
 We state our main theorem as follows

\begin{theorem}\label{theorem1}
Consider the $n-$body problem on $S^2$. The Lagrange action functional (\ref{functional}) reaches its minimum on $H$. This minimum is a periodic non-collision solution of the equations of motion (\ref{eq}).
\end{theorem}

We first show that any critical point of (\ref{functional}) on $D$ satisfies (\ref{eq}).

\begin{prop}
A critical point of the Lagrange action functional on $D$ is a solution of the equations of motion.
\end{prop}

\begin{proof}
Let $q_0=({q_0}_1,\dots,{q_0}_n)$ be a critical point of the Lagrange action functional on $D$.

%

For a given $q$, a displacement $\delta f$ is given by (the Gateaux derivative)

\begin{equation}
 \delta f=\dfrac{d}{d \varepsilon}\int_0^1\left( \dfrac{1}{2}\sum_1^nm_i|\dot{q}_i(t)+\varepsilon \dot{p}_i|^2+U(q(t)+\varepsilon p(t)) \right) dt\Big|_{\varepsilon=0} , \label{gateaux}
\end{equation}
restricted to any $p=(p_1,\dots,p_n)$ such that $|q_i(t)+\varepsilon p_i(t)|^2=1$, for every $\varepsilon \rightarrow 0$, and $i=1,\dots,n$.
Let $g_i$ be the function defined as $g_i(q_i)=|q_i(t)|^2-1$ (the constraint $g_i(q_i)=0$ maintains the particle $q_i$ on the sphere $S^2$). At a given time, for displacements  of the constraint equation, the following should be held

\begin{equation}
\delta g_i=\dfrac{d}{d \varepsilon}\left( |q_i(t)+\varepsilon p_i(t)|^2-1\right)\Big|_{\varepsilon=0}=0, \ \ i=1,\cdots,n.  \nonumber
\end{equation}

Integrating both sides with respect time we have

\begin{equation}
\delta h_i=\int_0^1 \dfrac{d}{d \varepsilon}\left( |q_i(t)+\varepsilon p_i(t)|^2-1\right)\Big|_{\varepsilon=0}dt=0,  \ \ i=1,\cdots,n.  \nonumber
\end{equation}

From Hamilton principle we have

\begin{equation}
\begin{split}
0=&\delta f+\sum_{i=1}^n\lambda_i \delta h_i\\
=&\dfrac{d}{d \varepsilon}\int_0^1\left( \dfrac{1}{2}\sum_{i=1}^nm_i|\dot{q}_i(t)+\varepsilon \dot{p}_i|^2+U(q(t)+\varepsilon p(t)) +\sum_{i=1}^n \lambda_i g_i \Big|_{q_i=q_{0_i}} \right) dt\Big|_{\varepsilon=0}\\
\end{split}  \nonumber
\end{equation}
where each $\lambda_i$ is the Lagrange multiplier corresponding to the body $i$, it will be computed later in the proof	.

Then we have
\begin{equation}
\begin{split}
0=&\dfrac{d}{d \varepsilon}\int_0^1\left( \dfrac{1}{2}\sum_1^nm_i|\dot{q}_i(t)+\varepsilon \dot{p}_i|^2+\dfrac{1}{2}\sum_{i=1}^n\sum_{j=1,j\neq i}^nm_im_j\cot(d(q_i+\varepsilon p_i,q_j+\varepsilon p_j))\right.\\
\end{split}  \nonumber
\end{equation}
\begin{equation}
\begin{split}
&\left. \left. +\lambda_i( |q_i(t)+\varepsilon p_i(t)|^2-1) \Big|_{q_i=q_{0_i}} \right) \Big|_{\varepsilon=0} dt \right]\\
=&\sum_{i=1}^n \left[ \int_0^1  \dfrac{d}{d \varepsilon}\left( \dfrac{1}{2}m_i |\dot{q}_i(t)+\varepsilon \dot{p}_i|^2 + \right. \right.\\
& \left. \left. \dfrac{1}{2}\sum_{j=1,j\neq i}^n\dfrac{m_im_j\dfrac{(q_i+\varepsilon p_i)\cdot (q_j+\varepsilon p_j)}{\sqrt{(q_i+\varepsilon p_i)\cdot (q_i+\varepsilon p_i)}\sqrt{(q_j+\varepsilon p_j)\cdot (q_j+\varepsilon p_j)}}}{\left(1-\left(\dfrac{(q_i+\varepsilon p_i)\cdot (q_j+\varepsilon p_j)}{\sqrt{(q_i+\varepsilon p_i)\cdot (q_i+\varepsilon p_i)}\sqrt{(q_j+\varepsilon p_j)\cdot (q_j+\varepsilon p_j)}}\right)^2\right)^{1/2}}  \right.  \right.\\
&\left. \left. +\lambda_i( |q_i(t)+\varepsilon p_i(t)|^2-1)  \Big|_{q_i=q_{0_i}}\right) \Big|_{\varepsilon=0} dt \right].\\
\end{split}   \nonumber
\end{equation}

After considering $\varepsilon \rightarrow 0$, and $q_i \cdot q_i=1$ we have
\begin{equation}
\begin{split}
0=&\sum_{i=1}^n \left[ \int_0^1 \left(m_i \dot{q}_i\cdot \dot{p}_i + \dfrac{1}{2}\sum_{j=1,j \neq i}^nm_im_j \dfrac{[q_i \cdot p_j+q_j \cdot p_i]-(q_i\cdot q_j)[q_j \cdot p_j+q_i\cdot p_i]}{(1-(q_i \cdot q_j)^2)^{3/2}} \right. \right.\\
&\left. \left. +2\lambda_i (q_i \cdot p_i) \Big|_{q_i=q_{0_i}} \right) dt \right]\\
=&\sum_{i=1}^n \left[ \int_0^1 \left( m_i\dot{q}_i\cdot \dot{p}_i +\sum_{j=1,j \neq i}^nm_im_j \dfrac{[q_j \cdot p_i]-(q_i\cdot q_j)[ q_i \cdot p_i]}{(1-(q_i \cdot q_j)^2)^{3/2}} +2\lambda_i (q_i \cdot p_i) \Big|_{q_i=q_{0_i}} \right) dt \right]\\
=&\sum_{i=1}^n \left[ \int_0^1 \left( m_i\dot{q}_i\cdot \dot{p}_i +\sum_{j=1,j \neq i}^nm_im_j \left(\dfrac{q_j -(q_i\cdot q_j) q_i}{(1-(q_i \cdot q_j)^2)^{3/2}} \right) \cdot p_i +2\lambda_i (q_i \cdot p_i) \Big|_{q_i=q_{0_i}} \right) dt \right].
\end{split}  \nonumber
\end{equation}

Integrating the first term and considering that the variations vanish at the the end points
\begin{equation}
\begin{split}
0=&\sum_{i=1}^n \left[m_i \dot{q}_i  \cdot p_i \Big|_0^1+ \int_0^1 \left( -m_i\ddot{q}_i\cdot p_i +\sum_{j=1,j \neq i}^nm_im_j \left(\dfrac{q_j -(q_i\cdot q_j) q_i}{(1-(q_i \cdot q_j)^2)^{3/2}} \right) \cdot p_i \right. \right. \\
&\left. \left.+2\lambda_i (q_i \cdot p_i) \Big|_{q_i=q_{0_i}} \right) dt \right]\\
=&\sum_{i=1}^n \left[ \int_0^1 \left( -m_i\ddot{q}_i\cdot p_i +\sum_{j=1,j \neq i}^nm_im_j \left(\dfrac{q_j -(q_i\cdot q_j) q_i}{(1-(q_i \cdot q_j)^2)^{3/2}} \right) \cdot p_i \right. \right. \\
&\left. \left.+2\lambda_i (q_i \cdot p_i) \Big|_{q_i=q_{0_i}} \right) dt \right]\\
=&\sum_{i=1}^n \left[ \int_0^1 \left( -m_i\ddot{q}_i +\sum_{j=1,j \neq i}^n m_im_j\left(\dfrac{q_j -(q_i\cdot q_j) q_i}{(1-(q_i \cdot q_j)^2)^{3/2}} \right) +\lambda_i q_i \right) \cdot p_i \ \Big|_{q_i=q_{0_i}} dt \right]\\
=&\sum_{i=1}^n \left[ \int_0^1 \left( -m_i\ddot{q}_i +\dfrac{\partial U}{\partial q_i} +2\lambda_i q_i \right) \cdot p_i \ \Big|_{q_i=q_{0_i}} dt \right].\\
\end{split}\nonumber
\end{equation}

Since this must hold for any $p=(p_1,\dots, p_n)$ in the interval $(0,1)$, it follows that the critical point  should satisfy

\begin{equation}
-m_i\ddot{q}_i +\dfrac{\partial U}{\partial q_i} -2\lambda_i q_i \Big|_{q_i=q_{0_i}}=0, \ \ \ i=1,\cdots,n, \label{eq4}
\end{equation}
where the multiplier $\lambda_i$ can be computed multiplying the last expression by $q_i$ 

\begin{equation}
-m_i\ddot{q}_i \cdot q_i +\dfrac{\partial U}{\partial q_i}\cdot q_i -2\lambda_i q_i \cdot q_i \Big|_{q_i=q_{0_i}}=0, \ \ \ i=1,\cdots,n. \nonumber
\end{equation}

Using the fact that the potential is a homogeneous function of degree zero, and that the expression $\ddot{q}_i\cdot q_i=-\dot{q}_i\cdot \dot{q}_i$ holds we have

\begin{equation}
\lambda_i=\dfrac{m_i\dot{q}_i\cdot \dot{q}_i}{2}. \nonumber
\end{equation}

Substituting this expression into (\ref{eq4}), we have

\begin{equation}
-m_i\ddot{q}_i  +\dfrac{\partial U}{\partial q_i} -m_i(\dot{q}_i\cdot \dot{q_i}) q_i \Big|_{q_i=q_{0_i}}=0, \ \ \ i=1,\cdots,n. \nonumber
\end{equation}

Hence, any critical point $q_0$ of the Lagrangian action satisfies the equation of motion.

\end{proof}

\section{Proof of Theorem \ref{theorem1}}
Now we prove that the action functional reaches its minimum on $D$. The proof of the theorem will be a consequence of the following result,

\begin{prop} \cite{Ramm}
A weakly lower semicontinuous from below functional $F(u)$, in a reflexive  Banach space $U$ is bounded from below on any  bounded  weakly  closed  set $M \subset DomF$ and attains its minimum  on $M$ at a point of $M$.
\end{prop}

Our task now is to prove that the functional ($\ref{functional}$) is weakly lower semicontinuous from below and that $D \cup \partial D$ is weakly closed.

\begin{prop}
$f(q)$ is weakly lower semicontinuous from below on $D \cup  \partial D$
\end{prop}

\begin{proof}
Recall that $f$ is called  weakly lower semicontinuous from below if for any  $q^n \in D \cup  \partial D$ such that $q^n \rightarrow q$ weakly, the following inequality holds

\[ \liminf_{n \rightarrow \infty} f(q^n)\geq f(q). \]

If $q \in D$, then there exists $N$ such that for $n>N$, $q^n \in D$. The functions $q^n_i$ are continuous and converges to $q_i$ uniformly.

This implies that $U(q^n_i)\rightarrow U(q_i)$ for $t \in [0,1]$.

By Fatou's lemma we have

\[ \liminf_{n \rightarrow \infty} f(q^n) \geq \int_0^1 \dfrac{1}{2}\sum_1^3|\dot{q}_i(t)|^2+ \int_0^1 \liminf_{n\rightarrow \infty}\left(\sum_{i<j} \cot d(q_i^n,q_j^n)\right)dt=f(q).\]

Now let us suppose that $q_i^n \in \partial D$ and $q_i^n \rightarrow q_i$ weakly.

There exist $t_0 \in [0,1)$ such that $q_{i_0}^n(t_0)=q_{j_0}^n(t_0)$ for $i_0\neq j_0$. Consider the set $C=\{ t \in [0,1) \ | \text{  there exist} \ i_0\neq j_0 \ \text{with} \ q_{i_0}(t)=q_{j_0}(t)  \}$.

Consider the Lebesgue measure, $\mu(C)$, of $C$. Firstly, let us suppose that $\mu(C)=0$. Since $q^n$ converges to $q$ uniformly, then the following holds almost everywhere,

$$\cot(d(q_i^n(t),q_j^n(t)))\rightarrow \cot(d(q_i(t),q_n(t))).$$

This implies, by Fatou's lemma
\begin{equation}
\begin{split}
\int_0^1 \cot(d(q_i(t_0),q_n(t_0)))&=\int_0^1\liminf_{n}\cot(d(q_i^n(t_0),q_j^n(t_0)))\\
&\leq \liminf_{n}  \int_0^1 \cot(d(q_i^n(t_0),q_j^n(t_0))).
\end{split}
\end{equation}

Hence $f(q)\leq \liminf_n f(q^ n)$. Secondly, if $\mu(C)>0$, then 

$$\int_0^1 \cot(d(q_i(t),q_j(t)))= +\infty.$$

Additionally we have,

$$\cot(d(q^n_i(t),q^n_j(t)))\rightarrow \cot(d(q_i(t),q_j(t))),$$
uniformly. This implies that

$$\int_0^1 \cot(d(q^n_i(t),q^n_j(t)))\rightarrow +\infty.$$ 

It follows that $$f(q)\leq \liminf_n f(q^n).$$

\end{proof}

\begin{prop}
$D \cup  \partial D$ is a weakly closed subset of $(W^{1,2}(\mathbb{R}/\mathbb{Z},S^2))^3:=\{(q_1,q_2,q_3) \in (S^2)^3 | \ q_i \in L^2, \ \dot{q}_i \in L^2, \ q_i(t+1)=q_i(t), \ i=1,2,3   \}$
\end{prop}

\begin{proof}

Since $q^n \rightarrow q$ weakly, then $q^n \rightarrow q$ uniformly, then $q \in D \cup \partial D$. Hence $D \cup  \partial D$ is a weakly closed subset of $(W^{1,2}(\mathbb{R}/\mathbb{Z},S^2))^3$.


 
\end{proof}

\section{Choreography solution for the three-problem on $S^2$}

In order to show a choreography solution for the three-body problem on $S^2$, we will firstly estimate the lower bound of the Lagrangian action for a binary collision generalized solution. We will consider masses equal to 1.

\begin{prop}\label{p}
Consider three bodies on $S^2$. Let $q \in T^*(S^2)^3$ be a periodic binary collision generalized solution, then the Lagrangian action satisfies $f(q)\geq \frac{3}{2}(12\pi)^{2/3}-3$.
\end{prop}

The following lemma will be useful to proof Proposition \ref{p}.

\begin{lemma}\label{cot}
Consider $q_i$ and $q_j$ on $S^2$ satisfying equations of motion (\ref{eq}), then

$$\dfrac{1}{r_{ij}}-1<\cot(d(q_i,q_j))<\dfrac{1}{r_{ij}},$$
where $r_{ij}$ is the Euclidean distance  between $q_i$ and $q_j$.
\end{lemma}

\begin{proof}
For this proof we will consider the origin of the system at the north pole  of the unit sphere, i.e., at $R=(0,0,1)$. The equations of motion takes the form

\begin{equation}
\ddot{q}_i=\sum_{i=1,j\neq i}^{n}\dfrac{q_j-\left(1-\frac{r_{ij}^2}{2}\right)q_i+\frac{r_{ij}^2R}{2}}{r_{ij}^2\left(1-\frac{r_{ij}^2}{4}\right)^{3/2}}-(\dot{q}_i\cdot \dot{q}_i)(q_i+R).
\end{equation}

The potential energy in $S^2$ is given by

\begin{equation}
\begin{split}
U=\sum_{i<j}\cot(d(q_i,q_j))&=\sum_{i<j}\dfrac{1-\frac{r_{ij}^2}{2}}{r_{ij}\left(1-\frac{r_{ij}^2}{4} \right)^{1/2}},
\end{split}
\end{equation}
for more details about the equations of motion and potential energy written in this coordinates, please see \cite{Diacu5}.

Consider $n=2$, then

\begin{equation}
\begin{split}
\cot(d(q_i,q_j))&=\dfrac{1-\frac{r_{ij}^2}{2}}{r_{ij}\left(1-\frac{r_{ij}^2}{4} \right)^{1/2}}> \dfrac{1-\frac{r_{ij}^2}{2}}{r_{ij}}=\left(\dfrac{1}{r_{ij}}-\dfrac{r_{ij}}{2}\right)> \dfrac{1}{r_{ij}}-1.
\end{split}
\end{equation}

On the other hand, we have

\begin{equation}
\begin{split}
\cot(d(q_i,q_j))&=\dfrac{1-\frac{r_{ij}^2}{2}}{r_{ij}\left(1-\frac{r_{ij}^2}{4} \right)^{1/2}}< \dfrac{1-\frac{r_{ij}^2}{4}}{r_{ij}\left(1-\frac{r_{ij}^2}{4} \right)^{1/2}}=\dfrac{\left(1-\frac{r_{ij}^2}{4} \right)^{1/2}}{r_{ij}}< \dfrac{1}{r_{ij}}.
\end{split}
\end{equation}

Hence we conclude the proof of the lemma.
\end{proof}

Now we can proceed with the proof of Proposition \ref{p}.

\begin{proof}

Consider three point particles $q_1, q_2, q_3 \in S^2$ with masses $m_1=m_2=m_3=1$ satisfying the equations of motion (\ref{eq}), and suppose that the particles $q_1$ and $q_2$ collide, without loss of generality, at the north pole.

The Lagrangian action is given by

$$f(q)=\int_0^1\left(\dfrac{1}{2}\sum_{i=1}^3 |\dot{q}_i^2|+\sum_{1\leq i<j\leq 3}\cot(d(q_i,q_j)) \right)dt,$$
where the constrains $|q_i|^2=1$ and $q_i\cdot \dot{q}_i=0$, $i=1,2,3$, hold.

Notice that \cite{Zhang1,Zhang2} $$\sum_{1\leq i<j\leq 3}|\dot{q}_i-\dot{q}_j|^2+\Bigl| \sum_{i=1}^3\dot{q}_i \Bigr|^2=3\sum_{i=1}^3|\dot{q}_i|^2.$$
We have 

\begin{equation}
\begin{split}
f(q)=&\int_0^1\left(\sum_{k=1}^3\dfrac{1}{2}|\dot{q}_k|^2+\sum_{1\leq i<j\leq 3}\cot(d(q_i,q_j)) \right)dt\\
\geq&\int_0^1\left(\sum_{1\leq i<j\leq 3}\dfrac{1}{6}|\dot{q}_i-\dot{q}_j|^2+\sum_{1\leq i<j\leq 3}\cot(d(q_i,q_j)) \right)dt \\
\geq&\int_0^1\left(\sum_{1\leq i<j\leq 3}\dfrac{1}{6}|\dot{q}_i-\dot{q}_j|^2+\sum_{1\leq i<j\leq 3}\dfrac{1}{r_{ij}}-3 \right)dt \ \text{(by Lemma \ref{cot})}\\
=&\dfrac{1}{3}\int_0^1\left(\sum_{1\leq i<j\leq 3}\dfrac{1}{2}|\dot{q}_i-\dot{q}_j|^2+\sum_{1\leq i<j\leq 3}\dfrac{3}{r_{ij}} \right)dt-3.\\
\end{split}
\end{equation}

If $q_1(t_0)=q_2(t_0)$, then $q_1(t_0+1/2)=q_2(t_0+1/2)$. Then using some estimates of \cite{Zhang1,Zhang2} we have

\begin{equation}
\begin{split}
\dfrac{1}{3}\int_0^1  \left( \dfrac{1}{2}|\dot{q}_1-\dot{q}_2|^2+\dfrac{3}{r_{12}} \right)dt&=\dfrac{2}{3}\int_0^{1/2}  \left( \dfrac{1}{2}|\dot{q}_1-\dot{q}_2|^2+\dfrac{3}{r_{12}} \right)dt\\
&=\dfrac{(12\pi)^{2/3}}{2}. \nonumber
\end{split}
\end{equation}

Since $q_1(t)=q_3(t+1/3)=q_2(t+2/3)$, then 

\begin{equation}
f(q)\geq \dfrac{3(12\pi)^{2/3}}{2}-3.
\end{equation}

\end{proof}

%
%
%

\begin{prop}
$f^{-1}((0,\frac{3}{2}(12 \pi )^{2/3}-3)) \neq \emptyset$
\end{prop}

\begin{proof}
Consider the test loop

\begin{equation}
q_1(t)=(x(t),y(t),z(t)), \ \ q_2(t)=q_1(t+1/3), \ \ q_3=q_1(t+2/3),\label{eq3}
\end{equation}
 where

\begin{equation}
\begin{split}
x(t)=&0.15 \sin(4 \pi t),\\
y(t)=&0.2275 \sin(2 \pi t),\\
z(t)=&\sqrt{1-x^2(t)-y^2(t)}. \nonumber
\end{split}
\end{equation}

In \cite{Zhang1} the authors show that if $\sin(2 \pi t)=\sin(2 \pi (t+\frac{i-1}{3}))$, then $\sin(4 \pi t)\neq \sin(4 \pi (t+\frac{i-1}{3}))$, for $t \in (0,1)$, $i=2,3$. Hence $q_i(t)\neq q_j(t)$, $i \neq j$.

With the expressions (\ref{eq3}), we have  $f(q)\approx 13.76572<\frac{3}{2}(12 \pi )^{2/3}-3\approx 13.8647$.

\end{proof}
%

\subsection*{Acknowledgements} The first author has been partially supported by {\it The 2017's Plan of Foreign Cultural and Educational Experts Recruitment for the Universities Under the Direct Supervision of Ministry of Education of China} (Grant WQ2017SCDX045).

\end{document}